\documentclass[11pt,bezier]{article}
\usepackage{amsmath,amssymb,amsfonts}

\usepackage[top=2.5cm,right=3cm,bottom=2cm,left=3cm]{geometry}

\usepackage{graphicx,xcolor}
\usepackage[pagebackref=true,colorlinks,linkcolor=red,citecolor=blue]{hyperref}
\usepackage{pgf,tikz}
\usepackage{mathrsfs}
\usepackage{float}
\usetikzlibrary{arrows}

\newtheorem{thm}{Theorem}[section]
\newtheorem{defi}[thm]{Definition}
\newtheorem{prop}[thm]{Proposition}
\newtheorem{remark}[thm]{Remark}
\newtheorem{cor}[thm]{Corollary}

\newtheorem{example}[thm]{Example}

\newenvironment{proof}[1]{{\bf proof. }{\rm #1}}{\hfill $\rule {2mm}{2mm}$\\}

\begin{document}
\title{Connected zero forcing sets and connected propagation time of graphs}
\author{{\sc M. Khosravi$^1$, S. Rashidi$^2$ and A. Sheikhhosseni$^3$ }\\
[5mm]
	 $^1,^3$Department of Pure Mathematics,\\
	  Faculty of Mathematics and Computer,\\
 Shahid Bahonar University of Kerman,\\
Kerman, Iran \\
{khosravi$ _- $m@uk.ac.ir  }\\
{sheikhhosseini@uk.ac.ir }\\
%\vspace{-0.25cm}
$^2$Department of Applied Mathematics,\\
  Faculty of Mathematics and Computer,\\
 Shahid Bahonar University of Kerman,\\
Kerman, Iran \\
{saeedeh.rashidi@uk.ac.ir}}
\date{\bf{ }}

\maketitle

%%%%%%%%%%%%%%%%%%%%%%%%%%%%%%%%%%%%%%%%%%%
\begin{abstract}
The zero forcing number $Z(G)$ of a graph $G$
is the minimum cardinality of a set $S$ with colored (black) vertices which forces the set $V(G)$
to be colored (black) after some times. ``color change rule'': a white vertex is changed to a black
vertex when it is the only white neighbor of a black vertex.
In this case, we say that the black vertex forces the white vertex.
We investigate here the concept of connected zero forcing set
and connected zero forcing number. We discusses this subject for special graphs and some products of graphs.
Also we introduce the connected propagation 
time. Graphs with extreme minimum connected propagation
times and maximum propagation times $|G|-1$ and $|G|-2$ are characterized.

\end{abstract}
\hspace*{-2.7mm} {\bf MSC 2010:} {\sf 05C50, 05C12}\\
\hspace*{-2.7mm} {\bf KEYWORDS:} { \sf zero forcing number; connected zero forcing number; propagation time}

\maketitle
%%%%%%%%%%%%%%%%%%%%%%%%%%%%%%%%%
\section{Introduction}
Let $V (G)$ and $E(G)$ be the vertex set and the edge set of a graph $G = (V,E)$, respectively.
For a vertex $v \in V(G)$, the $open\ neighborhood$ $N(v)$ is the set $\{u\in V (G) : uv \in E(G)\}$.
The $closed\ neighborhood$ $N[v]=N(v)\cup\{v\}$. Simple graph containing no graph loops or multiple edges.
In this paper each graph is undirected, finite,
simple and the vertex set is nonempty. A graph $G_0=(V_0, E_0)$
is a $subgraph$ of graph $G=(V,E)$ if $V'\subseteq V$ and $E'\subseteq E$. \\
A $complete\ graph$ is a graph in which every two distinct vertices are adjacent.
The complete graph on $n$ vertices is denoted by $K_n$.
A graph $(V,E)$ is $bipartite$ if the vertex set $V$ can be partitioned into two 
nonempty subsets $X$ and $Y$ such that every edge of graph has one 
vertex (endpoint) in $X$ and one in $Y$.
A $complete\ bipartite\ graph$ is bipartite graph $K_{n,m}$ with 
$|X|=n,\ |Y|=m$ and $E=\{\{u,w\}: u\in U, w\in W\}$.
The $path$ is the graph $P_n$ with $V(P_n) =\{ v_1,v_2,...,v_n\}$ and
$E(P_n) =\{ v_1v_2,v_2v_3,...,v_{n-1}v_n\}$.
The $cycle$ is the graph $C_n$ with $V(C_n) =\{ v_1,v_2,...,v_n\}$ 
and $E(C_n)=\{ v_1v_2,v_2v_3,...,v_{n-1}v_n, v_nv_1\}$.\\
The following graph operations are used to construct families of graphs:\\
$ \bullet $ The $Cartesian\ product$ of two  graphs $ G $ and $ H, $
is denoted by $ G  \square H, $ is the graph with vertex set $ V(G)
\times V(H)$ such that $ (u, v) $ is adjacent to $ (u', v' ) $ if
and only if $ (1)  u=u'$  and $ vv'  \in E(H), $ or   $ (2)  v=v'$
and $ uu' \in E(G). $ \\
$ \bullet $ The $strong\  product$ of two graphs $ G $ and $ H, $ is denoted by $ G  \boxtimes H, $
is the graph with vertex set $ V(G)  \times V(H)$ such that $ (u, v) $ is adjacent to
$ (u', v' ) $ if and only if $ (1)  u=u'$  and $ vv'  \in E(H), $ or   $ (2)  v=v'$  and
$ uu' \in E(G), $  or  $ (3)  uu' \in E(G)$  and $ vv' \in E(H). $ \\
$ \bullet $ The $corona\  product$ of two graphs $ G $ and $ H, $
denoted $ G \circ H, $ is the graph of order $ |G| |H|+|G| $
obtained by taking one copy of $ G $ and $|G|$ copies of $ H, $
and jointing all the vertices in the $i$th copy of $ H $ to the
$i$th vertex of $ G. $ \\
$ \bullet $ The $generalized\  corona$
	of $ G $ with $ H_{1}, H_{2}, \ldots, H_{n}, $  is defined as the
	graph obtained by joining all vertices of $ H_{i} $ to the $ i-
	$th vertex of $ G. $ We denote this graph by $ G \langle H_{1},
	H_{2}, \ldots, H_{n} \rangle. $\\
	$ \bullet $ The $n$th $supertraingle$ is denoted by $T_n$ is an equilateral triangular grid
	such that each side of it contains $n$ vertices (see Figure~\ref{1}). \\
$ \bullet $ The $wheel\ graph$ $W_n$ of order $n$ ($n$-wheel)  is a graph that contains a cycle of order $n-1$,
and for which every graph vertex in the cycle is connected to one other graph vertex.\\
$ \bullet$ The $star\ graph$ $S_n$ of order $n$ ($n$-star) complete bipartite graph $K_{1,n-1}$.\\
Zero forcing sets and the zero forcing number were introduced in~\cite{1}. The zero forcing number 
is a useful tool for determining the minimum rank of structured families of graphs and small graphs,
and is motivated by simple observations about null vectors of matrices.
\begin{defi}\label{d1}~\cite{1}
Color-change rule: Let $G$ be a graph with  each vertex colored
either white or black, $ u $ be a black vertex of $ G, $ and
exactly one neighbor $ v $ of $ u $ be white. Then change the
color of $ v $ to black. When this rule is applied, we say $ u $
forces $ v, $ and write $ u \rightarrow v. $
\end{defi}
\begin{defi}~\cite{1}
A $zero\  forcing\ set$ ( or $ZFS$ for brevity)  of a  graph $ G$ is a
subset $Z$ of vertices such that if initially the  vertices in  $ Z $ are  colored
  black and remaining vertices
 are colored white, the entire graph $ G $ may be colored black
  by repeatedly applying the
 color-change rule.
The zero forcing number of $ G,  Z(G), $ is the  minimum size of a
zero forcing set. Any zero forcing set of order  $Z(G)  $ is
called a minimum zero forcing set.
\end{defi}
For a coloring of $G$, the $derived\ coloring$
is the result of applying the color-change rule until no more changes are possible.
For the black set of vertices $B$, the derived coloring is denoted by $der(B)$ and it is unique~\cite{1}. 
\begin{example}
	In following figure, we see that $Z(T_4)=4$ and $Z(C_8)=2$.
	\begin{figure}[H]~\label{1}
		\vspace{-2.5cm}
		\hspace{1cm}
		\definecolor{ffffff}{rgb}{1.,1.,1.}
		\definecolor{qqqqff}{rgb}{0.,0.,1.}
		\begin{tikzpicture}[line cap=round,line join=round,>=triangle 45,x=1.0cm,y=1.0cm]
		\clip(-0.5,-1.2) rectangle (11.,4.6);
		\fill[color=ffffff] (3.03006,4.039) -- (0.,0.) -- (5.56,0.) -- cycle;
		\draw [color=ffffff] (3.03006,4.039)-- (0.,0.);
		\draw [color=ffffff] (0.,0.)-- (5.56,0.);
		\draw [color=ffffff] (5.56,0.)-- (3.03006,4.039);
		\draw(8.722369552238806,1.6474711194029854) circle (1.6898937223555441cm);
		\draw (2.04512,2.708)-- (3.874055635500463,2.6915773845283404);
		\draw (1.0651699175065268,1.4198468996682778)-- (4.6897987065389,1.3892594386781436);
		\draw (4.6897987065389,1.3892594386781436)-- (4.,0.);
		\draw (3.874055635500463,2.6915773845283404)-- (2.,0.);
		\draw (1.0651699175065268,1.4198468996682778)-- (2.,0.);
		\draw (2.04512,2.708)-- (4.,0.);
		\draw (9.81816,3.5066) node[anchor=north west] {l};
		\draw (10.45704,2.73462) node[anchor=north west] {b};
		\draw (10.64338,1.69644) node[anchor=north west] {j};
		\draw (7.55546,1.05756) node[anchor=north west] {h};
		\draw (10.03112,0.89784) node[anchor=north west] {i};
		\draw (8.35406,0.3122) node[anchor=north west] {c};
		\draw (7.36912,2.41518) node[anchor=north west] {g};
		\draw (7.76842,3.66632) node[anchor=north west] {a};
		\draw (2.65738,4.35844) node[anchor=north west] {a};
		\draw (1.69906,3.13392) node[anchor=north west] {g};
		\draw (0.74074,1.80292) node[anchor=north west] {h};
		\draw (-0.00462,0.76474) node[anchor=north west] {c};
		\draw (4.1481,3.32026) node[anchor=north west] {k};
		\draw (2.92358,2.3087) node[anchor=north west] {r};
		\draw (2.04512,0.8446) node[anchor=north west] {i};
		\draw (4.09486,0.1791) node[anchor=north west] {j};
		\draw (5.66544,0.07262) node[anchor=north west] {b};
		\draw (4.92008,1.74968) node[anchor=north west] {l};
		\draw (2.47104,-0.45978) node[anchor=north west] {$T_4$};
		\draw (8.88646,-0.32668) node[anchor=north west] {$C_8$};
		\draw (3.03006,4.039)-- (0.,0.);
		\draw (0.,0.)-- (5.56,0.);
		\draw (5.56,0.)-- (3.03006,4.039);
		\begin{scriptsize}
		\draw [fill=qqqqff] (0.,0.) circle (1.0pt);
		\draw [fill=qqqqff] (3.03006,4.039) circle (1.0pt);
		\draw [color=qqqqff] (5.56,0.) circle (1.0pt);
		\draw [fill=qqqqff] (9.41886,3.18716) circle (1.0pt);
		\draw [color=qqqqff] (7.15616,2.28208) circle (1.0pt);
		\draw [color=qqqqff] (10.03112,0.5784) circle (1.0pt);
		\draw [fill=qqqqff] (2.04512,2.708) circle (1.0pt);
		\draw [fill=qqqqff] (1.0651699175065268,1.4198468996682778) circle (1.0pt);
		\draw [color=qqqqff] (2.,0.) circle (1.0pt);
		\draw [color=qqqqff] (4.,0.) circle (1.0pt);
		\draw [color=qqqqff] (3.874055635500463,2.6915773845283404) circle (1.0pt);
		\draw [color=qqqqff] (4.6897987065389,1.3892594386781436) circle (1.0pt);
		\draw [color=qqqqff] (8.01495038059235,3.18216943126878) circle (1.0pt);
		\draw [fill=qqqqff] (10.244403971664957,2.381740837516689) circle (1.0pt);
		\draw [color=qqqqff] (10.37718,1.61658) circle (1.0pt);
		\draw [color=qqqqff] (7.286515776782339,0.7563680230308056) circle (1.0pt);
		\draw [color=qqqqff] (8.349817924124759,-8.450144882805599E-4) circle (1.0pt);
		\draw [fill=qqqqff] (2.9819314696692576,1.4102807231828292) circle (1.5pt);
		\end{scriptsize}
		\end{tikzpicture}
			\vspace{1.5cm}
		%	\caption{$Z(T_4)=4$ and $Z(C_8)=2$}
\end{figure}
\end{example}
\vspace{1cm}
\begin{defi}~\cite{10}
Let $B$ be a zero forcing set of $G$ and $B^0=B$.
For $t \geq 0$, we define $B^{(t+1)}$
to be the set of vertices $w$ for which there exists a vertex
$b\in\cup_{s=0}^{t}{B^{(s)}}$ such that $w$ is the only neighbor of $b$ not in $\cup_{s=0}^{t}{B^{(s)}}$. 
$The\ propagation\ time$ of $B$ in $G$, denoted $pt(G, B)$, is the smallest integer $t'$ such that
$V=\cup_{s=0}^{t'}{B^{(s)}}$.
\end{defi}
In the other word, the propagation time is the number of steps it take for an initial zero forcing
set to force all vertices of a graph to black.\\
It is possible that two minimum zero forcing set of one graph have the different 
propagation time (see Example 1.2 of~\cite{10}).  
Zero forcing parameters were studied and applied to the minimum rank problem and quantum systems~\cite{5,6,10}.
Also, this concept was named as graph infection or graph propagation~\cite{1}.
For more information about 
the propagation time, see~\cite{2,3,4,5,6,8,9,11,13}.
This parameter is investigated for graphs with extra condition or some
type of the zero forcing set~\cite{7,14}.
\begin{defi}~\cite{10} The minimum propagation time of $G$ is
$pt(G) = min\{pt(G, B) | B$ is a minimum zero forcing set of $G\}$.
\end{defi}
\begin{defi}~\cite{10} The maximum propagation time of G is defined as
$PT(G) = max\{pt(G, B) | B$ is a minimum zero forcing set of $G\}.$
 \end{defi}
 Note that $Z (G)$ and $pt (G)$ are not subgraph monotone.
 That is, a graph may have a subgraph with greater zero forcing number or minimum propagation time,
 as we see in the following example. Also, see the Example 1.5 of~\cite{10}.  
\begin{example}\label{10}
In Figure~\ref{1}, Let $ZFS(T_4)=\{a,g,h,c\}$, $ZFS(C_8)=\{a,g\}$ and $ZFS(P_6)=\{a\}$
 Let $B_0=ZFS(T_4)=\{a,g,h,c\}$, $B_1=ZFC(C_8)=\{a,g\}$ and $B_2=ZFC(P_6)=\{a\}$. Then we have:
\begin{center}
\begin{tabular}{|c|c|c|c|}
\hline
$i$&$B_{0}^{(i)}$&$B_{1}^{(i)}$&$B_{2}^{(i)}$\\
\hline
\hline
$0$&$B_0=\{a,g,h,c\}$&$B_1=\{a,g\}$&$B_2=\{a\}$\\
\hline
$1$&$\{k,i\}$&$\{h,l\}$&$\{g\}$\\
\hline
$2$&$\{r\}$&$\{c,b\}$&$\{h\}$\\
\hline
$3$&$\{l,j\}$&$\{i,j\}$&$\{c\}$\\
\hline
$4$&$\{b\}$&$--$&$\{i\}$\\
\hline
$5$&$--$&$--$&$\{j\}$\\
\hline
\end{tabular}
\end{center}
So, $pt(T_4,B_0)=4$, $pt(C_8,B_1)=3$ and $pt(P_6,B_2)=5$.
It is easy to see
that The minimum propagation time of $T_4$ is $4$,
the minimum propagation time of its subgraph $C_8$ is $3$
and the minimum propagation time of its subgraph $P_6$ is $5$.
So, this concept is not subgraph monotone 
(The cycle $C_8$ and the path $P_6$ are subgraphs of the graph $T_4$.).
\end{example}
%%%%%%%%%%%%%%%%%%%%%%%%%%%%%%%%%%%%%%%%%%%%%%%%%%%%%%%%%%%%%%%%%%%%%%%%%%%%%%%%%%%%%%%%%%%%%%%%%%%%%%%%
%%%%%%%%%%%%%%%%%%%%%%%%%%%%%%%%%%%%%%%%%%%%%%%%%%%%%%%%%%%%%%%%%%%%%%%%%%%%%%%%%%%%%%%%%%%%%%%%%%%%%%%%
\section{ Connected zero forcing set }
In this section, we investigate the concept of connected zero forcing.
We do a comparison between the zero forcing set and connected zero forcing set
for special graphs. In \cite{n}, it is mentioned that $R\subseteq V(G)$ is a connected forcing set of G,
If the subset $R$ is forcing set and it induces a connected subgraph.
Now we have the following definition.
\begin{defi}\label{d2}
	A $connected\  zero\  forcing\ set$ ( or $CZFS$ for brevity) of a
	graph $ G,  $ is a zero forcing set such that be connected in
	components  of $ G$. The $connected\ zero\ forcing\ number$ of $ G,
	Z_{c}(G), $ is the minimum size of connected zero forcing sets.
	Any connected zero forcing set of order  $Z_{c}(G)  $ is called a
	minimum connected zero forcing set.
\end{defi}
Applying Definition , we  obtain the following proposition.
\begin{prop}
Every CZFS is a ZFS and if a ZFS is connected in components, then it is a CZFS.
\end{prop}
	\vspace{-0.5cm}
\begin{cor}
For any graph $ G,  \, Z(G)  \leq Z_{c}(G).  $
\end{cor}
	\vspace{-0.25cm}
It may be interesting to  compare the zero forcing number and
connected zero forcing number
of some well known graphs:
	\vspace{-0.25cm}
\begin{example}~\label{11}
In Figure~\ref{81}, we see the ZFS and CZFS
for the star graph $S_6$. 
\end{example}
	\begin{figure}[H]~\label{81}
		\hspace{1.5cm}
	\definecolor{qqqqff}{rgb}{0.,0.,1.}
	\definecolor{xdxdff}{rgb}{0.49019607843137253,0.49019607843137253,1.}
	\begin{tikzpicture}[line cap=round,line join=round,>=triangle 45,x=1.0cm,y=1.0cm]
	\clip(-1.5,-1.5) rectangle (8.5,1.2);
	\draw (0.,0.)-- (0.,1.);
	\draw (0.,0.)-- (1.,0.);
	\draw (0.,0.)-- (0.,-1.);
	\draw (0.,0.)-- (-1.,0.);
	\draw (-0.76,0.78)-- (0.,0.);
	\draw (7.,0.)-- (7.,1.);
	\draw (7.,0.)-- (8.,0.);
	\draw (7.,0.)-- (7.,-1.);
	\draw (7.,0.)-- (6.,0.);
	\draw (7.,0.)-- (6.3,0.7);
	\draw (1.76,-0.44) node[anchor=north west] {ZFS};
	\draw (4.78,-0.5) node[anchor=north west] {CZFS};
	\begin{scriptsize}
	\draw [color=xdxdff] (0.,0.) circle (1.5pt);
	\draw [fill=xdxdff] (0.,1.) circle (1.5pt);
	\draw [fill=qqqqff] (1.,0.) circle (1.5pt);
	\draw [fill=qqqqff] (0.,-1.) circle (1.5pt);
	\draw [fill=qqqqff] (-1.,0.) circle (1.5pt);
	\draw [color=qqqqff] (-0.76,0.78) circle (1.5pt);
	\draw [fill=qqqqff] (7.,0.) circle (1.5pt);
	\draw [fill=qqqqff] (6.,0.) circle (1.5pt);
	\draw [fill=qqqqff] (8.,0.) circle (1.5pt);
	\draw [fill=qqqqff] (7.,1.) circle (1.5pt);
	\draw [fill=qqqqff] (7.,-1.) circle (1.5pt);
	\draw [color=qqqqff] (6.3,0.7) circle (1.5pt);
	\end{scriptsize}
	\end{tikzpicture}
\caption{ZFS and CZFS of star graph $S_6$}
	\end{figure}
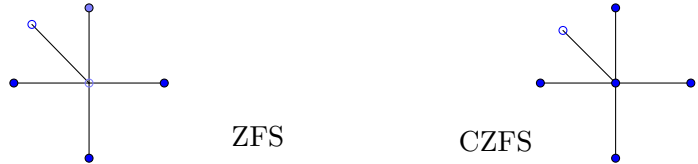
	\vspace{-0.25cm}
	From definition, we can easily observe that these zero forcing sets and 
	connected zero forcing sets are minimum. Thus, we have the following proposition
	for the star graphs and other well known graphs.
		\begin{prop}
				$ Z(K_{n})=Z_{c}(K_{n})=n-1; $
		 $ Z(P_{n})=Z_{c}(P_{n})=1; $
			 $Z(C_{n})=Z_{c}(C_{n})=2; $\\
			$Z(W_{n})=Z_{c}(W_{n})=3; $
			$Z(T_{n})=Z_{c}(T_{n})=n; $
		 $Z( S_{n})=n-2, \,  Z_{c}( S_{n})=n-1.$
		\end{prop}	
\begin{thm}
$Z(K_{n_{1}, n_{2}, \cdots, n_{t}})=Z_{c}(K_{n_{1},
n_{2}, \cdots, n_{t}})=n_{1}+ n_{2}
+ \cdots+ n_{t}-2.$
\end{thm}
\begin{proof}
Denote the parts of graph by $U_1$, $U_2\dots, U_t$. 
Let $v_1\in U_1$. If this vertex is black and has $n_2+n_3+\dots+n_t-1$
black neighbors, then it can change the color of another vertex. But the color of vertices $U_1-\{v_1\}$
can be changed by a vertex of another part for example $v_2\in U_2$. 
For this aim, all neighbors of $v_2$ must be black, except one of them. Therefore
$U_1-\{v_1\}$ at most has one white vertex that is,
$Z(K_{n_{1}, n_{2}, \cdots, n_{t}})= Z_{c}(K_{n_{1},
	n_{2}, \cdots, n_{t}})=n_{1}+ n_{2}
+ \cdots+ n_{t}-2.$
\end{proof}

Now, we find a bound for connected zero forcing number of different product 
of two graphs. First the strong product of a cycle and a path, we give a simple example.
The general case can be found by a similar way.  
\begin{example}~\label{20}
	In the following figure, we see a CZFS of size $n+2m-2=5+6-2=9$,
	for $C_{n} \boxtimes P_{m}, n=5$ and $m=3$. 
	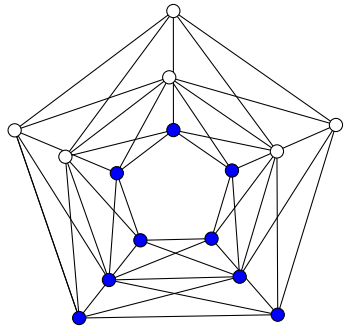
\begin{figure}[H]\label{6}
		\hspace{4.5cm}
		\definecolor{ffffff}{rgb}{1.,1.,1.}
		\definecolor{qqqqff}{rgb}{0.,0.,1.}
		\begin{tikzpicture}[line cap=round,line join=round,>=triangle 45,x=1.0cm,y=1.0cm]
		\clip(-0.9,1.4) rectangle (4.7,6.2);
		\fill[color=ffffff,fill=ffffff,fill opacity=0.10000000149011612] (1.338,2.99) -- (2.284,3.012) -- (2.555406833320207,3.918497838291464) -- (1.7771454810910725,4.456744313083895) -- (1.0247466799628067,3.882901090538967) -- cycle;
		\fill[color=ffffff,fill=ffffff,fill opacity=0.10000000149011612] (0.92,2.462) -- (2.658,2.506) -- (3.153225049506672,4.172532973073474) -- (1.721290962182145,5.158506993805295) -- (0.34108197705935517,4.101339477568479) -- cycle;
		\fill[color=ffffff,fill=ffffff,fill opacity=0.10000000149011612] (0.524,1.956) -- (3.164,2.) -- (3.9379583784328744,4.524385950771703) -- (1.776290962182145,6.040542269071334) -- (-0.3336513518668476,4.453192455266708) -- cycle;
		\draw [color=ffffff] (1.338,2.99)-- (2.284,3.012);
		\draw [color=ffffff] (2.284,3.012)-- (2.555406833320207,3.918497838291464);
		\draw [color=ffffff] (2.555406833320207,3.918497838291464)-- (1.7771454810910725,4.456744313083895);
		\draw [color=ffffff] (1.7771454810910725,4.456744313083895)-- (1.0247466799628067,3.882901090538967);
		\draw [color=ffffff] (1.0247466799628067,3.882901090538967)-- (1.338,2.99);
		\draw [color=ffffff] (0.92,2.462)-- (2.658,2.506);
		\draw [color=ffffff] (2.658,2.506)-- (3.153225049506672,4.172532973073474);
		\draw [color=ffffff] (3.153225049506672,4.172532973073474)-- (1.721290962182145,5.158506993805295);
		\draw [color=ffffff] (1.721290962182145,5.158506993805295)-- (0.34108197705935517,4.101339477568479);
		\draw [color=ffffff] (0.34108197705935517,4.101339477568479)-- (0.92,2.462);
		\draw [color=ffffff] (0.524,1.956)-- (3.164,2.);
		\draw [color=ffffff] (3.164,2.)-- (3.9379583784328744,4.524385950771703);
		\draw [color=ffffff] (3.9379583784328744,4.524385950771703)-- (1.776290962182145,6.040542269071334);
		\draw [color=ffffff] (1.776290962182145,6.040542269071334)-- (-0.3336513518668476,4.453192455266708);
		\draw [color=ffffff] (-0.3336513518668476,4.453192455266708)-- (0.524,1.956);
		\draw (1.338,2.99)-- (2.658,2.506);
		\draw (2.284,3.012)-- (0.92,2.462);
		\draw (1.338,2.99)-- (0.524,1.956);
		\draw (2.284,3.012)-- (3.164,2.);
		\draw (2.658,2.506)-- (0.524,1.956);
		\draw (0.92,2.462)-- (3.164,2.);
		\draw (2.555406833320207,3.918497838291464)-- (2.658,2.506);
		\draw (2.284,3.012)-- (3.153225049506672,4.172532973073474);
		\draw (3.153225049506672,4.172532973073474)-- (3.164,2.);
		\draw (2.658,2.506)-- (3.9379583784328744,4.524385950771703);
		\draw (2.555406833320207,3.918497838291464)-- (3.9379583784328744,4.524385950771703);
		\draw (1.776290962182145,6.040542269071334)-- (1.7771454810910725,4.456744313083895);
		\draw (1.0247466799628067,3.882901090538967)-- (-0.3336513518668476,4.453192455266708);
		\draw (0.34108197705935517,4.101339477568479)-- (1.338,2.99);
		\draw (1.0247466799628067,3.882901090538967)-- (0.92,2.462);
		\draw (0.34108197705935517,4.101339477568479)-- (0.524,1.956);
		\draw (-0.3336513518668476,4.453192455266708)-- (0.524,1.956);
		\draw (-0.3336513518668476,4.453192455266708)-- (0.92,2.462);
		\draw (1.721290962182145,5.158506993805295)-- (1.0247466799628067,3.882901090538967);
		\draw (0.34108197705935517,4.101339477568479)-- (1.7771454810910725,4.456744313083895);
		\draw (1.776290962182145,6.040542269071334)-- (0.34108197705935517,4.101339477568479);
		\draw (-0.3336513518668476,4.453192455266708)-- (1.721290962182145,5.158506993805295);
		\draw (1.721290962182145,5.158506993805295)-- (2.555406833320207,3.918497838291464);
		\draw (1.7771454810910725,4.456744313083895)-- (3.153225049506672,4.172532973073474);
		\draw (1.776290962182145,6.040542269071334)-- (3.153225049506672,4.172532973073474);
		\draw (1.721290962182145,5.158506993805295)-- (3.9379583784328744,4.524385950771703);
		\draw (1.776290962182145,6.040542269071334)-- (-0.3336513518668476,4.453192455266708);
		\draw (1.721290962182145,5.158506993805295)-- (0.34108197705935517,4.101339477568479);
		\draw (1.7771454810910725,4.456744313083895)-- (1.0247466799628067,3.882901090538967);
		\draw (1.7771454810910725,4.456744313083895)-- (2.555406833320207,3.918497838291464);
		\draw (1.721290962182145,5.158506993805295)-- (3.153225049506672,4.172532973073474);
		\draw (1.776290962182145,6.040542269071334)-- (3.9379583784328744,4.524385950771703);
		\draw (3.153225049506672,4.172532973073474)-- (2.658,2.506);
		\draw (2.555406833320207,3.918497838291464)-- (2.284,3.012);
		\draw (3.9379583784328744,4.524385950771703)-- (3.164,2.);
		\draw (3.164,2.)-- (0.524,1.956);
		\draw (2.658,2.506)-- (0.92,2.462);
		\draw (2.284,3.012)-- (1.338,2.99);
		\draw (1.338,2.99)-- (1.0247466799628067,3.882901090538967);
		\draw (0.92,2.462)-- (0.34108197705935517,4.101339477568479);
		\draw (0.524,1.956)-- (-0.3336513518668476,4.453192455266708);
		\begin{scriptsize}
		\draw [fill=qqqqff] (1.338,2.99) circle (2.5pt);
		\draw [fill=qqqqff] (2.284,3.012) circle (2.5pt);
		\draw [fill=qqqqff] (2.555406833320207,3.918497838291464) circle (2.5pt);
		\draw [fill=qqqqff] (1.7771454810910725,4.456744313083895) circle (2.5pt);
		\draw [fill=qqqqff] (1.0247466799628067,3.882901090538967) circle (2.5pt);
		\draw [fill=qqqqff] (0.92,2.462) circle (2.5pt);
		\draw [fill=qqqqff] (2.658,2.506) circle (2.5pt);
		\draw [fill=ffffff] (3.153225049506672,4.172532973073474) circle (2.5pt);
		\draw [fill=ffffff] (1.721290962182145,5.158506993805295) circle (2.5pt);
		\draw [fill=ffffff] (0.34108197705935517,4.101339477568479) circle (2.5pt);
		\draw [fill=qqqqff] (0.524,1.956) circle (2.5pt);
		\draw [fill=qqqqff] (3.164,2.) circle (2.5pt);
		\draw [fill=ffffff] (3.9379583784328744,4.524385950771703) circle (2.5pt);
		\draw [fill=ffffff] (1.776290962182145,6.040542269071334) circle (2.5pt);
		\draw [fill=ffffff] (-0.3336513518668476,4.453192455266708) circle (2.5pt);
		\end{scriptsize}
		\end{tikzpicture}
		\caption{$C_{n} \boxtimes P_{m}, n=5$ and $m=3$}
	\end{figure}
\end{example}
\begin{thm}
	$Z (C_{n} \boxtimes P_{m})$ and $Z_{c} (C_{n} \boxtimes
	P_{m})  \leq n+2m-2.$
\end{thm}
\begin{thm}
	\begin{itemize}
		\item $Z_{c}(G\square P_{t})= Z(G\square P_{t})=|G|.$
		\item $Z_{c}(G\square H)\leq min\{Z_{c}(G)|H|, Z_{c}(H)|G|\}$.
	\end{itemize}		
\end{thm}
\begin{proof}
	The proof is resulted from definitions.
\end{proof}
\begin{thm}
	$Z (P_{n} \boxtimes P_{m})= Z_{c} (P_{n} \boxtimes P_{m})
	\leq n+m-1.$
\end{thm}
\begin{proof}
	The graph $P_n \boxtimes P_m$ is a $n\times m$ grid.
	We know that $n+m-1\geq Z (P_{n} \boxtimes P_{m})$~\cite{1}. Now consider the CZFS 
	contains all vertices of the first column and the first row of this grid.
	This set is connected and ZFS.
	The size of this set is $n+m-1$.
\end{proof}
	\begin{thm}
Let $ G $ be a connected graph with $ \vert G \vert=n. $  Then
$$  Z_{c}(G\langle H_{1}, H_{2}, \ldots, H_{n}\rangle) \leq
\vert G \vert + \sum_{i=1}^{n} Z(H_{i}). $$
In particular, if $H_i$'s are connected and $ \vert H_{i} \vert > 1$ for all $ i, $ then
$$  Z_{c}(G\langle H_{1}, H_{2}, \ldots, H_{n}\rangle) = \vert G
\vert + \sum_{i=1}^{n} Z(H_{i}). $$
\end{thm}
\begin{proof}
Let $ Z_{i} $ be a minimum zero forcing set of $ H_{i}$ for all $1\leq i\leq n$.
It is easy, to see that 
$V(G) \cup Z_{1} \cup \ldots \cup Z_{n} $ is a ZFS for $ G \langle
H_{1}, H_{2}, \ldots, H_{n} \rangle$. Also this set is connected, because 
the graph $G$ is connected. Thus
$$  Z_{c}(G\langle H_{1}, H_{2}, \ldots, H_{n}\rangle) \leq \vert G \vert +
\sum_{i=1}^{n} Z(H_{i}). $$
Now, we claim that if $ \vert H_{i}
\vert > 1 $ for all $ i, $ then this set is minimum. First notice
that every CZFS contains all vertices of $ G. $ Because if a vertex of G such as $v_i$ is initially white,
then by connectedness of zero forcing set, all of the vertices of $H_i$ must be white or the CZFS is a subset of $V (G)$.
In the first case, forcing of vertices of $H_i$ should be begin with a force by $v_i$, which is not possible where $|H_i|>1$.
In the second case, there exists an vertex $v_i$ such that it is the white neighbor of each vertex in $H_i$.
All vertices in $H_i$
should be black and the rest of graph is a path. Thus, it easily follows that $|G|=1$ and $H_i$ is complete graph.
So the result follows.\\
Now, if $v_i\rightarrow u$ is a force on this graph, it means that $u$ is the only
white vertex in $H_i$. The graph $H_i$ is connected, so there exists another vertex
$u_0\in H_i$ which is a neighbor of $u$. Thus we can replace $v_i\rightarrow u$ by $u_0\rightarrow u$.
Therefore, all the forces are in $H_i$. So, for every minimum CZFS $A$ of
$G\langle H_{1}, H_{2}, \ldots, H_{n}\rangle$, $A\cap V (H_i)$ is a ZFS for $H_i$. Now, it follows that
$V (G) \cup Z_1 \cup\cdots \cup Z_n$ is minimum.
\end{proof}
\begin{cor}
For any graph $ G$ and $ H, $ we have
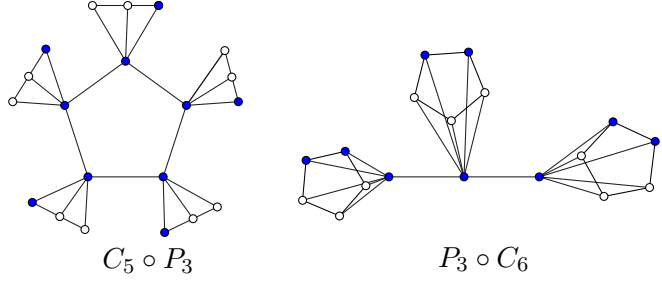
\begin{itemize}
\item $ Z(G \circ H) \leq  Z(G) + \vert G \vert Z(H). $
\item $ Z_c(G \circ H) \leq  Z_c(G) + \vert G \vert Z(H).$
\end{itemize}
\end{cor}
\begin{figure}[H]
	\vspace{-1.5cm}
	\hspace{1.5cm}
	\definecolor{sqsqsq}{rgb}{0.12549019607843137,0.12549019607843137,0.12549019607843137}
	\definecolor{ffffff}{rgb}{1.,1.,1.}
	\definecolor{qqqqff}{rgb}{0.,0.,1.}
	\begin{tikzpicture}[line cap=round,line join=round,>=triangle 45,x=1.0cm,y=1.0cm]
	\clip(-4.8,-5.58) rectangle (16.48,6.72);
	\fill[color=ffffff,fill=ffffff,fill opacity=1.0] (-2.,3.) -- (-1.,3.) -- (-0.6909830056250525,3.951056516295153) -- (-1.5,4.538841768587627) -- (-2.3090169943749475,3.9510565162951536) -- cycle;
	\fill[color=ffffff,fill=ffffff,fill opacity=1.0] (4.56,3.28) -- (4.86,2.74) -- (5.466275617111868,2.8584477779260746) -- (5.540974555037319,3.471652530576289) -- (4.980865420486898,3.7321861318510186) -- cycle;
	\fill[color=ffffff,fill=ffffff,fill opacity=1.0] (3.06,4.66) -- (2.48,4.62) -- (2.3388124039143374,4.056026540773814) -- (2.8315536707435047,3.7474717742191777) -- (3.277272117389275,4.120747900323809) -- cycle;
	\fill[color=ffffff,fill=ffffff,fill opacity=1.0] (1.42,3.34) -- (0.9,3.22) -- (0.8534379448804457,2.6883685722015276) -- (1.3446610122305143,2.479802280334435) -- (1.6948156190303902,2.882532650851514) -- cycle;
	\draw [color=ffffff] (-2.,3.)-- (-1.,3.);
	\draw [color=ffffff] (-1.,3.)-- (-0.6909830056250525,3.951056516295153);
	\draw [color=ffffff] (-0.6909830056250525,3.951056516295153)-- (-1.5,4.538841768587627);
	\draw [color=sqsqsq] (-1.4706666666666668,5.28)-- (-1.5,4.538841768587627);
	\draw [color=ffffff] (-2.3090169943749475,3.9510565162951536)-- (-2.,3.);
	\draw [color=sqsqsq] (-0.18,4.68)-- (0.,4.);
	\draw [color=sqsqsq] (-2.56,4.7)-- (-3.,4.);
	\draw [color=sqsqsq] (-2.74,2.66)-- (-2.04,2.3);
	\draw [color=sqsqsq] (-0.28,2.6)-- (-0.98,2.26);
	\draw [color=sqsqsq] (-2.56,4.7)-- (-2.3090169943749475,3.9510565162951536);
	\draw [color=sqsqsq] (-2.787852545348153,4.337507314218848)-- (-2.3090169943749475,3.9510565162951536);
	\draw [color=sqsqsq] (-3.,4.)-- (-2.3090169943749475,3.9510565162951536);
	\draw [color=sqsqsq] (-2.74,2.66)-- (-2.,3.);
	\draw [color=sqsqsq] (-2.,3.)-- (-2.04,2.3);
	\draw [color=sqsqsq] (-2.,3.)-- (-2.3744092963202066,2.4719819238218204);
	\draw [color=sqsqsq] (-1.,3.)-- (-0.98,2.26);
	\draw [color=sqsqsq] (-1.,3.)-- (-0.6022589167767503,2.4434742404227214);
	\draw [color=sqsqsq] (-1.,3.)-- (-0.28,2.6);
	\draw [color=sqsqsq] (-1.94,5.28)-- (-1.5,4.538841768587627);
	\draw [color=sqsqsq] (-0.18,4.68)-- (-0.6909830056250525,3.951056516295153);
	\draw [color=sqsqsq] (-0.18,4.68)-- (-0.6909830056250525,3.951056516295153);
	\draw [color=sqsqsq]  (-0.08672594987873877,4.327631366208569) -- (-0.6909830056250525,3.951056516295153);
	\draw [color=sqsqsq] (0.,4.)-- (-0.6909830056250525,3.951056516295153);
	\draw [color=sqsqsq] (-1.06,5.28)-- (-1.5,4.538841768587627);
	\draw [color=sqsqsq] (-1.94,5.28)-- (-1.06,5.28);
	\draw [color=sqsqsq] (-1.5,4.538841768587627)-- (-2.3090169943749475,3.9510565162951536);
	\draw [color=sqsqsq] (-2.3090169943749475,3.9510565162951536)-- (-2.,3.);
	\draw [color=sqsqsq] (-2.,3.)-- (-1.,3.);
	\draw [color=sqsqsq] (-1.,3.)-- (-0.6909830056250525,3.951056516295153);
	\draw [color=sqsqsq] (-0.6909830056250525,3.951056516295153)-- (-1.5,4.538841768587627);
	\draw [color=sqsqsq] (2.,3.)-- (3.,3.);
	\draw [color=sqsqsq] (3.,3.)-- (4.,3.);
	\draw [color=ffffff] (4.56,3.28)-- (4.86,2.74);
	\draw [color=ffffff] (4.86,2.74)-- (5.466275617111868,2.8584477779260746);
	\draw [color=ffffff] (5.466275617111868,2.8584477779260746)-- (5.540974555037319,3.471652530576289);
	\draw [color=ffffff] (5.540974555037319,3.471652530576289)-- (4.980865420486898,3.7321861318510186);
	\draw [color=ffffff] (4.980865420486898,3.7321861318510186)-- (4.56,3.28);
	\draw [color=ffffff] (3.06,4.66)-- (2.48,4.62);
	\draw [color=ffffff] (2.48,4.62)-- (2.3388124039143374,4.056026540773814);
	\draw [color=ffffff] (2.3388124039143374,4.056026540773814)-- (2.8315536707435047,3.7474717742191777);
	\draw [color=ffffff] (2.8315536707435047,3.7474717742191777)-- (3.277272117389275,4.120747900323809);
	\draw [color=ffffff] (3.277272117389275,4.120747900323809)-- (3.06,4.66);
	\draw [color=ffffff] (1.42,3.34)-- (0.9,3.22);
	\draw [color=ffffff] (0.9,3.22)-- (0.8534379448804457,2.6883685722015276);
	\draw [color=ffffff] (0.8534379448804457,2.6883685722015276)-- (1.3446610122305143,2.479802280334435);
	\draw [color=ffffff] (1.3446610122305143,2.479802280334435)-- (1.6948156190303902,2.882532650851514);
	\draw [color=ffffff] (1.6948156190303902,2.882532650851514)-- (1.42,3.34);
	\draw [color=sqsqsq] (2.8315536707435047,3.7474717742191777)-- (3.,3.);
	\draw [color=sqsqsq] (3.277272117389275,4.120747900323809)-- (3.,3.);
	\draw [color=sqsqsq] (2.3388124039143374,4.056026540773814)-- (3.,3.);
	\draw [color=sqsqsq] (3.06,4.66)-- (3.,3.);
	\draw [color=sqsqsq] (2.48,4.62)-- (3.,3.);
	\draw [color=sqsqsq] (1.6948156190303902,2.882532650851514)-- (2.,3.);
	\draw [color=sqsqsq] (1.3446610122305143,2.479802280334435)-- (2.,3.);
	\draw [color=sqsqsq] (1.42,3.34)-- (2.,3.);
	\draw [color=sqsqsq] (0.9,3.22)-- (2.,3.);
	\draw [color=sqsqsq] (0.8534379448804457,2.6883685722015276)-- (2.,3.);
	\draw [color=sqsqsq] (4.56,3.28)-- (4.,3.);
	\draw [color=sqsqsq] (4.980865420486898,3.7321861318510186)-- (4.,3.);
	\draw [color=sqsqsq] (4.86,2.74)-- (4.,3.);
	\draw [color=sqsqsq] (5.466275617111868,2.8584477779260746)-- (4.,3.);
	\draw [color=sqsqsq] (5.540974555037319,3.471652530576289)-- (4.,3.);
	\draw (-1.94,2.18) node[anchor=north west] {$C_5\circ P_3$};
	\draw (2.52,2.2) node[anchor=north west] {$P_3\circ C_6$};
	\draw (2.48,4.62)-- (3.06,4.66);
	\draw (3.06,4.66)-- (3.277272117389275,4.120747900323809);
	\draw (3.277272117389275,4.120747900323809)-- (2.8315536707435047,3.7474717742191777);
	\draw (2.8315536707435047,3.7474717742191777)-- (2.3388124039143374,4.056026540773814);
	\draw (2.3388124039143374,4.056026540773814)-- (2.48,4.62);
	\draw (1.42,3.34)-- (1.6948156190303902,2.882532650851514);
	\draw (1.6948156190303902,2.882532650851514)-- (1.3446610122305143,2.479802280334435);
	\draw (1.3446610122305143,2.479802280334435)-- (0.8534379448804457,2.6883685722015276);
	\draw (0.8534379448804457,2.6883685722015276)-- (0.9,3.22);
	\draw (0.9,3.22)-- (1.42,3.34);
	\draw (4.980865420486898,3.7321861318510186)-- (4.56,3.28);
	\draw (4.56,3.28)-- (4.86,2.74);
	\draw (4.86,2.74)-- (5.466275617111868,2.8584477779260746);
	\draw (5.466275617111868,2.8584477779260746)-- (5.540974555037319,3.471652530576289);
	\draw (5.540974555037319,3.471652530576289)-- (4.980865420486898,3.7321861318510186);
	\begin{scriptsize}
	\draw [fill=qqqqff] (-2.,3.) circle (1.5pt);
	\draw [fill=qqqqff] (-1.,3.) circle (1.5pt);
	\draw [fill=qqqqff] (-0.6909830056250525,3.951056516295153) circle (1.5pt);
	\draw [fill=qqqqff] (-1.5,4.538841768587627) circle (1.5pt);
	\draw [fill=qqqqff] (-2.3090169943749475,3.9510565162951536) circle (1.5pt);
	\draw [fill=ffffff] (-0.18,4.68) circle (1.5pt);
	\draw [fill=qqqqff] (0.,4.) circle (1.5pt);
	\draw [fill=ffffff] (-1.94,5.28) circle (1.5pt);
	\draw [fill=qqqqff] (-2.56,4.7) circle (1.5pt);
	\draw [fill=ffffff] (-3.,4.) circle (1.5pt);
	\draw [fill=qqqqff] (-2.74,2.66) circle (1.5pt);
	\draw [fill=ffffff] (-2.04,2.3) circle (1.5pt);
	\draw [fill=ffffff] (-0.28,2.6) circle (1.5pt);
	\draw [fill=qqqqff] (-0.98,2.26) circle (1.5pt);
	\draw [fill=ffffff] (-2.787852545348153,4.337507314218848) circle (1.5pt);
	\draw [fill=ffffff] (-2.3744092963202066,2.4719819238218204) circle (1.5pt);
	\draw [fill=ffffff] (-0.6022589167767503,2.4434742404227214) circle (1.5pt);
	\draw [fill=ffffff] (-0.08672594987873877,4.327631366208569) circle (1.5pt);
	\draw [fill=qqqqff] (-1.06,5.28) circle (1.5pt);
	\draw [fill=qqqqff] (-1.06,5.28) circle (1.5pt);
	\draw [fill=ffffff] (-1.4706666666666668,5.28) circle (1.5pt);
	\draw [fill=qqqqff] (2.,3.) circle (1.5pt);
	\draw [fill=qqqqff] (3.,3.) circle (1.5pt);
	\draw [fill=qqqqff] (4.,3.) circle (1.5pt);
	\draw [fill=ffffff] (4.56,3.28) circle (1.5pt);
	\draw [fill=ffffff] (4.86,2.74) circle (1.5pt);
	\draw [fill=ffffff] (5.466275617111868,2.8584477779260746) circle (1.5pt);
	\draw [fill=qqqqff] (5.540974555037319,3.471652530576289) circle (1.5pt);
	\draw [fill=qqqqff] (4.980865420486898,3.7321861318510186) circle (1.5pt);
	\draw [fill=qqqqff] (3.06,4.66) circle (1.5pt);
	\draw [fill=qqqqff] (2.48,4.62) circle (1.5pt);
	\draw [fill=ffffff] (2.3388124039143374,4.056026540773814) circle (1.5pt);
	\draw [fill=ffffff] (2.8315536707435047,3.7474717742191777) circle (1.5pt);
	\draw [fill=ffffff] (3.277272117389275,4.120747900323809) circle (1.5pt);
	\draw [fill=qqqqff] (1.42,3.34) circle (1.5pt);
	\draw [fill=qqqqff] (0.9,3.22) circle (1.5pt);
	\draw [fill=ffffff] (0.8534379448804457,2.6883685722015276) circle (1.5pt);
	\draw [fill=ffffff] (1.3446610122305143,2.479802280334435) circle (1.5pt);
	\draw [fill=ffffff] (1.6948156190303902,2.882532650851514) circle (1.5pt);
	\end{scriptsize}
	\end{tikzpicture}
	\vspace{-7cm}
\end{figure}
\begin{remark}
	The integer of $  Z_{c}(G)-Z (G)$ can  be very large, for example if $ G=C_{n}   $ and $ H= P_{m}$  where
	$ m > 1, $ then $ Z(G \circ H )=n+2 $ and $ Z_{c}(G \circ H)=2n $ while if we set $
	G=P_{n}  $ and $ H= C_{m},\ m>2$ then $ Z(G \circ H)=2n+1$ and $ Z_{c}(G
	\circ H)=3n$.
	\end{remark}
%%%%%%%%%%%%%%%%%%%%%%%%%%%%%%%%%%%%%%%%%%%%%%%%%%%%%%%%%%%%%%%%%%%%%%%%%%%%%%%%%%%%%%%%%%%%%%%%%%%%%%%%%%
%%%%%%%%%%%%%%%%%%%%%%%%%%%%%%%%%%%%%%%%%%%%%%%%%%%%%%%%%%%%%%%%%%%%%%%%%%%%%%%%%%%%%%%%%%%%%%%%%%%%%%%%%% 
\section{Connected propagation time}
In this section, we introduce the concept of connected propagation time for graph $G$.
We characterize graphs $G$ having extreme minimum and maximum connected propagation time
 $|G|-1$, $|G|-2$ and $0$.
\begin{defi}
	Let $B$ be a connected zero forcing set of $G$ and $B^0=B$.
	for $t \geq 0$, $B^{(t+1)}$
	is the set of vertices $w$ for which there exists a vertex
	$b\in\cup_{s=0}^{t}{B^{(s)}}$ such that $w$ is the only neighbor of $b$ not in $\cup_{s=0}^{t}{B^{(s)}}$. 
	The connected propagation time of $B$ in $G$, denoted by $pt_{c}(G, B)$, is the smallest integer $t'$ such that
	$V=\cup_{s=0}^{t'}{B^{(s)}}$.
\end{defi}
\begin{defi} The minimum connected propagation time of $G$ is
	$pt_{c}(G) = min\{pt_{c}(G, B) | B$ is a minimum connected zero forcing set of $G\}$.
\end{defi}
\begin{defi} The maximum connected propagation time of G is defined as
	$PT_{c}(G) = max\{pt_{c}(G, B) | B$ is a minimum connected zero forcing set of $G\}.$
\end{defi}
Here, the case of connected propagation time $|G|-1$ and $|G|-2$ are investigated.
\begin{thm}\label{t1}Let $ G $ be a connected graph.
Then the following conditions are equivalent:
\begin{itemize}
\item $ Z_{c} (G)=1;$
\item $  pt_{c}(G)=\vert G \vert -1; $
\item $ Pt_{c}(G)=\vert G \vert -1; $
\item $ G $ is a path.
\end{itemize}
\end{thm}
\begin{proof}
We prove the sequence $ (ii)\rightarrow(iii)\rightarrow(i)\rightarrow(iv)
 \rightarrow(ii). $\\
First, note that for each graph $ G, $ we have $ pt_{c} (G) \leq Pt_{c}(G)
\leq \vert G\vert -Z_{c}(G) \leq \vert G\vert-1. $  Thus, from $ (ii), $ we can
deduce $ (iii), $ and if $ (iii) $ holds, then $ Z_{c}(G)=1. $ In addition,
if $ Z_{c}(G)=1, $ then
 $Z(G)=1 $ and therefore $ G $ is a path.  The last part is trivial.
\end{proof}
\begin{defi}
For $k\geq1$ and nonnegative integer numbers $n_1,\cdots,n_k$
($n_i$s can be zero), by $PC(n_1,n_2,\cdots,n_k)$ we mean a
collection of graphs which contains a path
$P_{k+2}=(v_1,\cdots,v_{k+2})$ and $k$ cycles, with the following
conditions:
\begin{itemize}
\item the $i$th cycle ($1\leq i\leq k$) is $(v_i,v_{i+1},v_{i+2},u^{i}_{1},\dots,u^{i}_{n_i})$ with $n_i$
new vertices, which are probably neighbors of  $v_{i+1}$.
\item For each $i,j,s,t$ we have $u^i_j\neq u^t_s$ and $u^i_j\neq v_t$.

\end{itemize}

For instance, the following figures are PC(3) and PC(4,3,5,0,2).

\vspace*{-1.5cm} \hspace*{-2cm}
\definecolor{qqqqff}{rgb}{0.,0.,1.}
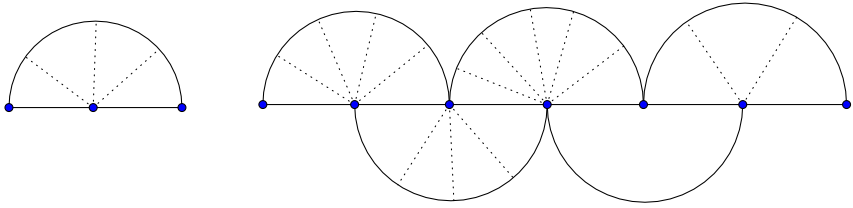
\begin{tikzpicture}[line cap=round,line join=round,>=triangle 45,x=1.0cm,y=1.0cm]
\clip(-4.3,-2.48) rectangle (7.06,6.3); \draw (-1.44,2.4)--
(-0.32,2.4); \draw (-0.32,2.4)-- (0.86,2.4); \draw
[shift={(-0.29,2.4)}]
plot[domain=0.:3.141592653589793,variable=\t]({1.*1.15*cos(\t
r)+0.*1.15*sin(\t r)},{0.*1.15*cos(\t r)+1.*1.15*sin(\t r)});
\draw [dotted] (-0.32,2.4)--
(-1.220930458496562,3.075180332535968); \draw [dotted]
(-0.32,2.4)-- (-0.2797325521063168,3.5499541640925303); \draw
[dotted] (-0.32,2.4)-- (0.5573133587217656,3.1775346115329146);
\begin{scriptsize}
\draw [fill=qqqqff] (-1.44,2.4) circle (1.5pt); \draw
[fill=qqqqff] (-0.32,2.4) circle (1.5pt); \draw [fill=qqqqff]
(0.86,2.4) circle (1.5pt);
\end{scriptsize}
\end{tikzpicture}
%\caption{PC(4,3,5,2)}
%\end{figure}
%\begin{figure}
\hspace*{-7cm}
\begin{tikzpicture}[line cap=round,line join=round,>=triangle 45,x=1.0cm,y=1.0cm]
\clip(-4.3,-2.48) rectangle (7.06,6.3); \draw (-2.7,2.44)--
(-1.48,2.44); \draw (-1.48,2.44)-- (-0.22,2.44); \draw
(-0.22,2.44)-- (1.08,2.44); \draw (1.08,2.44)-- (2.36,2.44); \draw
(2.36,2.44)-- (3.68,2.44); \draw [shift={(-1.46,2.44)}]
plot[domain=0.:3.141592653589793,variable=\t]({1.*1.24*cos(\t
r)+0.*1.24*sin(\t r)},{0.*1.24*cos(\t r)+1.*1.24*sin(\t r)});
\draw [shift={(1.07,2.44)}]
plot[domain=0.:3.141592653589793,variable=\t]({1.*1.29*cos(\t
r)+0.*1.29*sin(\t r)},{0.*1.29*cos(\t r)+1.*1.29*sin(\t r)});
\draw [shift={(2.38,2.44)}]
plot[domain=3.141592653589793:6.283185307179586,variable=\t]({1.*1.3*cos(\t
r)+0.*1.3*sin(\t r)},{0.*1.3*cos(\t r)+1.*1.3*sin(\t r)}); \draw
[shift={(-0.2,2.44)}]
plot[domain=3.141592653589793:6.283185307179586,variable=\t]({1.*1.28*cos(\t
r)+0.*1.28*sin(\t r)},{0.*1.28*cos(\t r)+1.*1.28*sin(\t r)});
\draw [dotted] (-1.48,2.44)-- (-2.52,3.1); \draw [dotted]
(-1.48,2.44)-- (-1.9654870420636226,3.572290974222514); \draw
[dotted] (-1.48,2.44)-- (-1.1931798615145965,3.6509529362029847);
\draw [dotted] (-1.48,2.44)--
(-0.5057434992786792,3.231829862300671); \draw [dotted]
(-0.22,2.44)-- (-0.9053587644693887,1.3718852995177824); \draw
[dotted] (-0.22,2.44)-- (-0.16063401499267416,1.1606054872618832);
\draw [dotted] (-0.22,2.44)--
(0.6295918457962982,1.4652295811893494); \draw [dotted]
(1.08,2.44)-- (-0.1263432516746661,2.922558622524235); \draw
[dotted] (1.08,2.44)-- (0.2058857288597623,3.397813408974721);
\draw [dotted] (1.08,2.44)--
(0.8545993806433108,3.7118893714394954); \draw [dotted]
(1.08,2.44)-- (1.4312,3.6784); \draw [dotted] (1.08,2.44)--
(2.0983944834364734,3.2187841719227666); \draw (3.68,2.44)--
(5.06,2.44); \draw [shift={(3.71,2.44)}]
plot[domain=0.:3.141592653589793,variable=\t]({1.*1.35*cos(\t
r)+0.*1.35*sin(\t r)},{0.*1.35*cos(\t r)+1.*1.35*sin(\t r)});
\draw [dotted] (3.68,2.44)--
(4.396784571370965,3.6022508130893254); \draw [dotted]
(3.68,2.44)-- (2.9302536191295805,3.5420415516301924);
\begin{scriptsize}
\draw [fill=qqqqff] (-2.7,2.44) circle (1.5pt); \draw[fill=qqqqff]
(-1.48,2.44) circle (1.5pt); \draw [fill=qqqqff] (-0.22,2.44)
circle (1.5pt); \draw [fill=qqqqff](1.08,2.44) circle (1.5pt);
\draw[fill=qqqqff] (2.36,2.44) circle (1.5pt); \draw[fill=qqqqff]
 (3.68,2.44) circle (1.5pt); \draw[fill=qqqqff]
(5.06,2.44) circle (1.5pt);
\end{scriptsize}
\end{tikzpicture}
\vspace{-3.5cm}
\end{defi}
\begin{defi} Let $G$ and $H$ be disjoint graphs, each with a vertex
labeled $v$. Then $G\oplus_{ v} H$ is the graph obtained by
identifying the vertex $v$ in $G$ with the vertex $v$ in $H$.
\end{defi}
\begin{thm}\label{t2}
Let $ G $ be a  graph with $ Pt_{c}(G)=\vert G \vert -2 $. Then
\begin{enumerate}
\item If  $ G $ is not connected, then $G=K_1\dot{\cup} P_{|G|-1}$.
\item If $G$ is connected, then $G$ is
of the form $PC(n_1,n_2,\cdots,n_{k-1},0)$ or $PC(n_1,n_2,\cdots,n_k)\oplus_{v_{k+1}}P_m $ where
$(v_1,\dots, v_{k+2}) $ is the path of PC and 
$P_m$ is a path with $m\geq2$ and pendent vertex $v_{k+1}$.
\end{enumerate}
\end{thm}
\begin{proof}
From $Pt_{c}(G)  \leq \vert G \vert -Z_{c}(G), $ it  follows
that $ Z_{c}(G) \leq 2. $ If $ Z_{c}(G)=1 $ then by previous
theorem $ Pt_{c}(G)=\vert G \vert-1, $ which is a contradiction. So, $Z_{c}(G)=2. $\\
If $G$ is not connected then it contain two component,
each of them has the CZFS of size one thus each component is one path and in every stage,
only the color of one vertex will be changed. So, one path is $P_1=K_1$ and the graph $G$ is $G=K_1\dot{\cup} P_{|G|-1}$. \\
If $G$ is connected, then the connected zero forcing set is consist of two vertices $ v_{1} $ and $v_{2} $
which are adjacent by an edge $e$. Observe that in each stage, only
the color of one vertex will be changed.
The edge $ e $ is not a cutting-
edge, otherwise $ G\setminus \lbrace e \rbrace $ has two components $
G_{1} $ and $ G_{2} $ with $ v_{1}  \in V( G_{1})$ and $ v_{2} \in
V( G_{2}),$ thus $ \lbrace v_{1} \rbrace$ is  a zero forcing set
for $ G_{1} $ and $ \lbrace v_{2} \rbrace$ is  a zero forcing set
for $ G_{2}. $ Therefore,  $ G_{1} $ and $ G_{2} $ are path and so
$ G $ is a path and $ \lbrace v_{1}, v_{2} \rbrace $ is
not an minimum connected zero forcing set. Thus, there is another path
between  $ v_{1} $ and $ v_{2} $.\\
Consider a path between $v_1$
and $v_2$ with largest possible length, namely,
$P: v_1,u_1,u_2,\cdots, u_n, v_2$.
Just one vertex of $v_1$ or $v_2$ can have neighborers in $V(G)-V(P)$. Otherwise 
$v_1$ and $v_2$ can not change the color of vertices $V(G)-\{v_1,v_2\}$.
Because both of them have at least two white neighbors.\\
 Without loss of generality, let
$v_1$ has no neighbors except $u_1$ and $v_2$. There exist two
cases for $v_2$.
\begin{enumerate}
\item[Case 1.] If $v_2$ has another neighbors which does not belong to
this path, then in the first step $v_1$ forces $u_1$. Again, $u_1$
can not have more than one white neighbor, which is $u_2$. 
So, the only neighbors of $u_1$ are $v_1,u_2$ and probably $v_2$.
Continuing this way, in the $i$th step, $u_{i-1}$ forces $u_i$ and
all of the $u_i$s with $i<n$ are adjacent to $u_{i-1}, u_{i+1}$
and probably $v_2$. After $n$ steps, the set $\{u_n,v_2\}$ is a
zero forcing set for $G\setminus\{v_1,u_1,\cdots,u_{n-1}\}$ and
since the path $P$ has maximum length, it follows that there is no
other path except the edge $\{u_n,v_2\}$ between these two vertices 
$G\setminus\{v_1,u_1,\cdots,u_{n-1}\}$.
This means that this edge, is a cutting edge.
 Since in each step only one vertices can be forced by black
vertices, it can be deduced that $u_n$ has no other neighbors and
a path can be jointed to $v_2$. Putting $v_3=u_{n}$, $G$ is of the
form $PC(n-1)\oplus_{v_2}P_m$ with $m\geq2$ (see following figure).\\
		\definecolor{xdxdff}{rgb}{0.49019607843137253,0.49019607843137253,1.}
		\definecolor{qqqqff}{rgb}{0.,0.,1.}
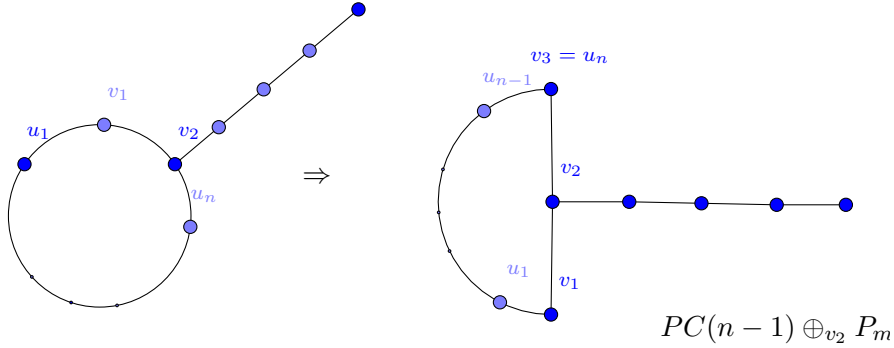
		\begin{tikzpicture}[line cap=round,line join=round,>=triangle 45,x=1.0cm,y=1.0cm]
		\clip(-2.5,1.2) rectangle (9.5,6.4);
		\draw(-1.,3.3125) circle (1.213530489934225cm);
		\draw (0.,4.)-- (2.44,6.06);
		\draw [shift={(5.,3.5)}] plot[domain=1.5707963267948966:4.71238898038469,variable=\t]({1.*1.5*cos(\t r)+0.*1.5*sin(\t r)},{0.*1.5*cos(\t r)+1.*1.5*sin(\t r)});
		\draw (5.,5.)-- (5.02,3.5);
		\draw (5.02,3.5)-- (5.,2.);
		\draw (5.02,3.5)-- (6.04,3.5);
		\draw (6.04,3.5)-- (7.,3.48);
		\draw (7.,3.48)-- (8.,3.46);
		\draw (8.,3.46)-- (8.92,3.46);
		\draw (1.56,4.02) node[anchor=north west] {$\Rightarrow$};
		\draw (6.32,2.14) node[anchor=north west] {$PC(n-1)\oplus_{v_2}P_m$};
		\begin{scriptsize}
		\draw [fill=qqqqff] (-2.,4.) circle (2.5pt);
		\draw[color=qqqqff] (-1.81,4.42) node {$u_1$};
		\draw [fill=qqqqff] (0.,4.) circle (2.5pt);
		\draw[color=qqqqff] (0.19,4.42) node {$v_2$};
		\draw [fill=qqqqff] (-1.38,2.16) circle (0.5pt);
		\draw [fill=qqqqff] (2.44,6.06) circle (2.5pt);
		\draw [fill=xdxdff] (0.5832722708194403,4.492434786019691) circle (2.5pt);
		\draw [fill=xdxdff] (1.1800400109834075,4.996263287961401) circle (2.5pt);
		\draw [fill=xdxdff] (1.7903032204919007,5.511485505825129) circle (2.5pt);
		\draw [fill=xdxdff] (-0.9418127850751751,4.524634686418672) circle (2.5pt);
		\draw[color=xdxdff] (-0.75,4.94) node {$v_1$};
		\draw [fill=xdxdff] (-1.9019592587082141,2.500634404208105) circle (0.5pt);
		\draw [fill=xdxdff] (0.2047360243166858,3.1666666646009016) circle (2.5pt);
		\draw[color=xdxdff] (0.39,3.58) node {$u_n$};
		\draw [fill=xdxdff] (-0.7692232350775088,2.121114951087639) circle (0.5pt);
		\draw [fill=qqqqff] (5.,2.) circle (2.5pt);
		\draw[color=qqqqff] (5.25,2.42) node {$v_1$};
		\draw [fill=qqqqff] (5.,5.) circle (2.5pt);
		\draw[color=qqqqff] (5.25,5.42) node {$v_3=u_n$};
		\draw [fill=qqqqff] (5.02,3.5) circle (2.5pt);
		\draw[color=qqqqff] (5.27,3.92) node {$v_2$};
		\draw [fill=qqqqff] (6.04,3.5) circle (2.5pt);
		\draw [fill=qqqqff] (7.,3.48) circle (2.5pt);
		\draw [fill=qqqqff] (8.,3.46) circle (2.5pt);
		\draw [fill=qqqqff] (8.92,3.46) circle (2.5pt);
		\draw [fill=xdxdff] (4.109526577144665,4.70708619542612) circle (2.5pt);
		\draw[color=xdxdff] (4.43,5.12) node {$u_{n-1}$};
		\draw [fill=xdxdff] (3.563260572168273,3.9310218283495177) circle (0.5pt);
		\draw [fill=xdxdff] (3.5066663973281864,3.3587387132607756) circle (0.5pt);
		\draw [fill=xdxdff] (3.6507377621360035,2.844644055894631) circle (0.5pt);
		\draw [fill=xdxdff] (4.320961460334153,2.1624998461127287) circle (2.5pt);
		\draw[color=xdxdff] (4.57,2.58) node {$u_1$};
	\end{scriptsize}
	\end{tikzpicture}
\item[Case 2.] All of the neighbors of $v_2$ belongs to the path.
If $n=1$, put $v_3=u_1$ and $G=PC(0)=K_3$, else let $1\leq n_1<n$,
be the largest number for which $u_{n_1}$ is adjacent to $v_2$.
Put $v_3=u_{n_1}$. As previous part, it can be seen that each
$u_i\ (1\leq i\leq i_2)$ are at most of degree 3 with two
neighbors of the path and probably $v_2$ and in the $i$th step,
$u_{i-1}$ forces $u_i$.\\
Now $\{v_2,v_3\}$ is a CZFS for $G\setminus\{v_1,u_1,\cdots,
u_{{n_1}-1}\}$ such that in each step only one vertex becomes black.
So by repeating the proof, there exist two cases, either $G$ is of the form
$PC(n_1-1,n-n_1)\oplus_{v_3} P_m$ with $v_4=u_{n_1+1}$ or all neighbors of $v_3$ belongs to the path.
In the first case the graph has form $A$ (see following figure).
	\definecolor{xdxdff}{rgb}{0.49019607843137253,0.49019607843137253,1.}
	\definecolor{qqqqff}{rgb}{0.,0.,1.}
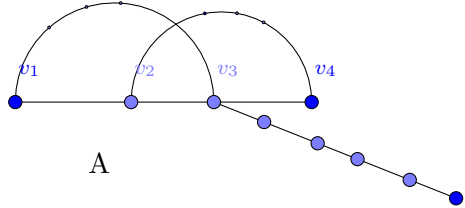
	\begin{tikzpicture}[line cap=round,line join=round,>=triangle 45,x=1.0cm,y=1.0cm]
	\clip(0.3,1.9) rectangle (10.8,5.3);
	\draw (3.88,3.4)-- (7.82,3.4);
	\draw [shift={(5.2,3.4)}] plot[domain=0.:3.141592653589793,variable=\t]({1.*1.32*cos(\t r)+0.*1.32*sin(\t r)},{0.*1.32*cos(\t r)+1.*1.32*sin(\t r)});
	\draw [shift={(6.62,3.4)}] plot[domain=0.:3.141592653589793,variable=\t]({1.*1.2*cos(\t r)+0.*1.2*sin(\t r)},{0.*1.2*cos(\t r)+1.*1.2*sin(\t r)});
	\draw (6.52,3.4)-- (9.74,2.12);
	\draw (4.72,2.86) node[anchor=north west] {A};
	\begin{scriptsize}
	\draw [fill=qqqqff] (3.88,3.4) circle (2.5pt);
	\draw[color=qqqqff] (4.07,3.82) node {$v_1$};
	\draw [fill=qqqqff] (7.82,3.4) circle (2.5pt);
	\draw[color=qqqqff] (8.01,3.82) node {$v_4$};
	\draw [fill=xdxdff] (5.42,3.4) circle (2.5pt);
	\draw[color=xdxdff] (5.61,3.82) node {$v_2$};
	\draw [fill=xdxdff] (6.52,3.4) circle (2.5pt);
	\draw[color=xdxdff] (6.71,3.82) node {$v_3$};
	\draw [fill=xdxdff] (4.329341344660464,4.3921459095729585) circle (0.5pt);
	\draw [fill=xdxdff] (4.826416941488345,4.666031476067273) circle (0.5pt);
	\draw [fill=xdxdff] (5.281077394972196,4.717507668298186) circle (0.5pt);
	\draw [fill=qqqqff] (6.4,4.58) circle (0.5pt);
	\draw [fill=xdxdff] (6.827359373609751,4.581948429575577) circle (0.5pt);
	\draw [fill=xdxdff] (7.180545239714663,4.461032060888469) circle (0.5pt);
	\draw [fill=qqqqff] (9.74,2.12) circle (2.5pt);
	\draw [fill=xdxdff] (7.187527220630371,3.1346475644699145) circle (2.5pt);
	\draw [fill=xdxdff] (7.899340974212034,2.8516905444126075) circle (2.5pt);
	\draw [fill=xdxdff] (8.42939541547278,2.6409856733524357) circle (2.5pt);
	\draw [fill=xdxdff] (9.12414040114613,2.3648137535816627) circle (2.5pt);
	\end{scriptsize}
	\end{tikzpicture}\\
When all neighbors of $v_3$ belong to the path,
if $n-n_1=1$, then put
$v_4=u_{n_1+1}$ and $G=PC(n_1-1,0)$ and graph is of the form $B$ (see following figure).
Else we can choose $v_4=u_{n_2}$, such that $n_1+1<n_2\leq n$ and $n_2$ is the
smallest number for which $u_{n_2}$ is adjacent to $v_3$. 
  \definecolor{xdxdff}{rgb}{0.49019607843137253,0.49019607843137253,1.}
  \definecolor{qqqqff}{rgb}{0.,0.,1.}
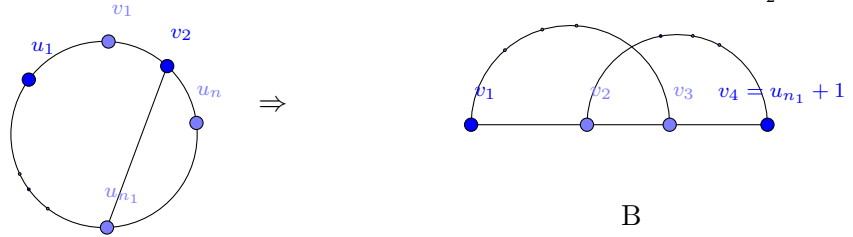
  \begin{tikzpicture}[line cap=round,line join=round,>=triangle 45,x=1.0cm,y=1.0cm]
  \clip(-3.,1.5) rectangle (9.,5.);
  \draw(-0.9997826086956524,3.27) circle (1.238278979013887cm);
  \draw (0.92,3.86) node[anchor=north west] {$\Rightarrow$};
  \draw (-0.16,4.18)-- (-0.9610126742620914,2.032328104039002);
  \draw (3.88,3.4)-- (7.82,3.4);
  \draw [shift={(5.2,3.4)}] plot[domain=0.:3.141592653589793,variable=\t]({1.*1.32*cos(\t r)+0.*1.32*sin(\t r)},{0.*1.32*cos(\t r)+1.*1.32*sin(\t r)});
  \draw [shift={(6.62,3.4)}] plot[domain=0.:3.141592653589793,variable=\t]({1.*1.2*cos(\t r)+0.*1.2*sin(\t r)},{0.*1.2*cos(\t r)+1.*1.2*sin(\t r)});
  \draw (5.74,2.46) node[anchor=north west] {B};
  \begin{scriptsize}
  \draw [fill=qqqqff] (-2.,4.) circle (2.5pt);
  \draw[color=qqqqff] (-1.81,4.42) node {$u_1$};
  \draw [fill=qqqqff] (-0.16,4.18) circle (2.5pt);
  \draw[color=qqqqff] (0.03,4.6) node {$v_2$};
  \draw [fill=qqqqff] (-2.,2.54) circle (0.5pt);
  \draw [fill=xdxdff] (-0.9404087357798766,4.50685470977098) circle (2.5pt);
  \draw[color=xdxdff] (-0.75,4.92) node {$v_1$};
  \draw [fill=xdxdff] (-2.1210381685115953,2.744528022203361) circle (0.5pt);
  \draw [fill=xdxdff] (0.2289342273794156,3.4235896045093837) circle (2.5pt);
  \draw[color=xdxdff] (0.41,3.84) node {$u_n$};
  \draw [fill=xdxdff] (-1.7478702930185108,2.2832377964108383) circle (0.5pt);
  \draw [fill=xdxdff] (-0.9610126742620914,2.032328104039002) circle (2.5pt);
  \draw[color=xdxdff] (-0.75,2.46) node {$u_{n_1}$};
  \draw [fill=qqqqff] (3.88,3.4) circle (2.5pt);
  \draw[color=qqqqff] (4.07,3.82) node {$v_1$};
  \draw [fill=qqqqff] (7.82,3.4) circle (2.5pt);
  \draw[color=qqqqff] (8.01,3.82) node {$v_4=u_{n_1}+1$};
  \draw [fill=xdxdff] (5.42,3.4) circle (2.5pt);
  \draw[color=xdxdff] (5.61,3.82) node {$v_2$};
  \draw [fill=xdxdff] (6.52,3.4) circle (2.5pt);
  \draw[color=xdxdff] (6.71,3.82) node {$v_3$};
  \draw [fill=xdxdff] (4.329341344660464,4.3921459095729585) circle (0.5pt);
  \draw [fill=xdxdff] (4.826416941488345,4.666031476067273) circle (0.5pt);
  \draw [fill=xdxdff] (5.281077394972196,4.717507668298186) circle (0.5pt);
  \draw [fill=qqqqff] (6.4,4.58) circle (0.5pt);
  \draw [fill=xdxdff] (6.827359373609751,4.581948429575577) circle (0.5pt);
  \draw [fill=xdxdff] (7.180545239714663,4.461032060888469) circle (0.5pt);
  \end{scriptsize}
  \end{tikzpicture}
 Going on, we obtain a path $v_1,v_2,\cdots, v_{k+2}$ with
$v_i=u_{n_{i-2}}$ and $1\leq n_1<n_3<\dots<n_4<n_2<n$ and $G$ is
of the form $PC(n_1-1,n-n_2,n_3-n_1-1,n_2-n_4-1,...,
n_{k-1}-n_{k-3}-1, |n_k-n_{k-2}|-1)\oplus_{v_{k+1}} P_m$ or
$|n_k-n_{k-2}|=1$. Notice that $n_1-1$ is the number of vertices between $v_1$ and $v_3$.
$n-n_2=n-(n_1+1)$ is the number of vertices between $v_2$ and $v_4$, in the fact $i$th
integer is the number between $v_i$ and $v_{i+2}$. So, $G$ is of the form $PC(n_1-1,n-n_2,n_3-n_1-1,n_2-n_4-1,...,
n_{k-1}-n_{k-3}-1, |n_k-n_{k-2}|-1,0)$.
\end{enumerate}
\end{proof}
\begin{thm}\label{pt}
For connected graph $ G, $  we have $ pt_{c}(G)= \vert G  \vert-2
$ if and only if $ G $ is as Theorem \ref{t2} and in
addition\begin{enumerate}
    \item If $G=PC(n)\oplus_{v_{2}}P_m $ for $m\geq1$,  with cycle
    $v_1,u_1,\cdots,u_n,v_3,v_2$, then $u_1$ and $u_n$ are
    adjacent to $v_2$.
    \item If $G=PC(n_1,n_2,\cdots,n_k)\oplus_{v_{k+1}}P_m $ for some $k>1$ and $m\geq1$, then in
    the first cycle $v_1,u_1,\cdots,u_n,v_3,v_2$, we have $u_1$ is
    adjacent to $v_2$.
\end{enumerate}
\end{thm}
\begin{proof}
If $ pt_{c}(G)= \vert G \vert-2, $ then  $ Pt_{c}(G)  \geq \vert G
\vert-2. $  From Theorem \ref{t1}, it follows that $ Pt_{c}(G)  =
\vert G \vert-2. $ Thus, $G$ needs to be of one of the forms
described in Theorem \ref{t2}.

One should note that if $u_1$ in cycle
$v_1,u_1,\cdots,u_{n_1},v_3,v_2$, (first cycle) not to be a neighbor of $v_2$ then the
set $\{v_1,u_1\}$ is a CZFS of $G$ and in the first step two
vertices $u_2$ and $v_2$ become black. So, $ pt_{c}(G)< \vert
G \vert-2 $.\\
Similarly, if $k=1$, the set $\{v_3,u_n\}$ is another CZFS of $G$.
So, the vertex $u_n$ must be adjacent to $v_2$.

Conversely, if $G$ has the above conditions and $k>1$, we can
easily seen that even all the vertices of $i$th cycle with $1<i$
can not be a ZFS. Except in the case that $m=1$, in the $k$th
cycle each of the $\{v_{k+1},v_{k+2}\}$ or $\{v_k,v_{k+2}\}$ is a
CZFS with propagation time $|G|-2$.
In the case that $k=1$, if
$m=1$, then $G=K_3$.\\
 Let $m>1$ and $u_1$ and $u_n$ are
    adjacent to $v_2$. Then for each $1\leq i\leq n-1$, the set
    $\{u_i,u_{i+1}\}$ can not be a ZFS. To see this, let $1\leq j_1
    \leq i<i+1\leq j_2\leq n$ such that $j_1$ is the largest number
    not greater than $i$ such that $u_{j_1}$ is adjacent to $v_2$
    and $j_2$ is the smallest number
    greater than $i$ for which $u_{j_2}$ is adjacent to $v_2$. Then
    it is easily seen that $\{u_{j_1},u_{j_1+1},\cdots, u_{j_2}\}$
    can not be a ZFS (see following figure, case $A$).\\
    	In addition, for each $i$, if $u_i$ is a neighbor of $v_2$, then
    	the set
    	$\{u_i,v_2\}$ is not a ZFS (see following figure, case $B$). Because both of them have
    	degree greater than two.
    	Therefore, the only minimum CZFS's are
    	$\{v_1,v_2\},\ \{u_1,v_1\},\ \{v_3,v_2\},$ and $\{u_n,v_3\}$
    	with propagation times $|G|-2$.
    	
    	 \definecolor{xdxdff}{rgb}{0.49019607843137253,0.49019607843137253,1.}
    	 \definecolor{qqqqff}{rgb}{0.,0.,1.}
    	 \begin{tikzpicture}[line cap=round,line join=round,>=triangle 45,x=1.0cm,y=1.0cm]
    	 \clip(-2.,1.8) rectangle (11.,5.7);
    	 \draw [shift={(1.,3.)}] plot[domain=0.:3.141592653589793,variable=\t]({1.*2.*cos(\t r)+0.*2.*sin(\t r)},{0.*2.*cos(\t r)+1.*2.*sin(\t r)});
    	 \draw (-1.,3.)-- (3.,3.);
    	 \draw (-0.7092321511481208,4.038520800697592)-- (1.0266666666666668,3.);
    	 \draw (2.7232612753838783,4.015071710157627)-- (1.0266666666666668,3.);
    	 \draw (0.74,2.24) node[anchor=north west] {$A$};
    	 \draw [shift={(8.,3.)}] plot[domain=0.:3.141592653589793,variable=\t]({1.*2.*cos(\t r)+0.*2.*sin(\t r)},{0.*2.*cos(\t r)+1.*2.*sin(\t r)});
    	 \draw (6.,3.)-- (10.,3.);
    	 \draw (8.069725031470085,4.998784235475779)-- (8.02,3.);
    	 \draw (7.72,2.24) node[anchor=north west] {$B$};
    	 \begin{scriptsize}
    	 \draw [color=qqqqff] (-1.,3.) circle (2.5pt);
    	 \draw[color=qqqqff] (-0.81,3.42) node {$v_1$};
    	 \draw [color=qqqqff] (3.,3.) circle (2.5pt);
    	 \draw[color=qqqqff] (3.19,3.42) node {$v_3$};
    	 \draw [fill=xdxdff] (-0.7092321511481208,4.038520800697592) circle (2.5pt);
    	 \draw [fill=xdxdff] (-0.38146478860058997,4.446220950566243) circle (0.5pt);
    	 \draw [fill=xdxdff] (-0.054190127448115355,4.699612654457164) circle (0.5pt);
    	 \draw [fill=xdxdff] (1.7664390335318725,4.847314593640923) circle (0.5pt);
    	 \draw [fill=xdxdff] (2.061832752326009,4.694848431597284) circle (0.5pt);
    	 \draw [fill=xdxdff] (2.341294422537344,4.483552921897366) circle (0.5pt);
    	 \draw [fill=xdxdff] (2.7232612753838783,4.015071710157627) circle (2.5pt);
    	 \draw[color=xdxdff] (2.93,4.44) node {$u_{j_2}$};
    	 \draw [color=qqqqff] (1.0266666666666668,3.) circle (2.5pt);
    	 \draw[color=qqqqff] (1.21,3.42) node {$v_2$};
    	 \draw [fill=xdxdff] (0.2674703377838824,4.861021304008512) circle (0.5pt);
    	 \draw [fill=xdxdff] (1.3543059966305304,4.9683666479473905) circle (2.5pt);
    	 \draw[color=xdxdff] (1.63,5.38) node {$u_{i+1}$};
    	 \draw [fill=xdxdff] (0.8048571519372588,4.99045705023996) circle (2.5pt);
    	 \draw[color=xdxdff] (0.99,5.42) node {$u_i$};
    	 \draw [color=qqqqff] (6.,3.) circle (2.5pt);
    	 \draw[color=qqqqff] (6.25,3.42) node {$v_1$};
    	 \draw [color=qqqqff] (10.,3.) circle (2.5pt);
    	 \draw[color=qqqqff] (10.25,3.42) node {$v_3$};
    	 \draw [fill=xdxdff] (6.61853521139941,4.446220950566243) circle (0.5pt);
    	 \draw [fill=xdxdff] (6.945809872551885,4.699612654457164) circle (0.5pt);
    	 \draw [fill=xdxdff] (8.766439033531872,4.847314593640923) circle (0.5pt);
    	 \draw [fill=xdxdff] (9.061832752326008,4.6948484315972845) circle (0.5pt);
    	 \draw [fill=xdxdff] (9.341294422537345,4.483552921897366) circle (0.5pt);
    	 \draw [fill=xdxdff] (8.069725031470085,4.998784235475779) circle (2.5pt);
    	 \draw[color=xdxdff] (8.31,5.42) node {$u_i$};
    	 \draw [fill=qqqqff] (8.02,3.) circle (2.5pt);
    	 \draw[color=qqqqff] (8.27,3.42) node {$v_2$};
    	 \draw [fill=xdxdff] (7.267470337783881,4.861021304008512) circle (0.5pt);
    	 \end{scriptsize}
    	 \end{tikzpicture}
    	\end{proof}    	
    	\begin{remark}
    		In \cite{14}, Row investigated the graph with $Z(G)=2$. He
    		proved that $Z(G)=2$ if and only if $G$ is a graph of two parallel
    		paths. In his proof, he used a result of \cite{3} for
    		characterization the graphs with maximum nullity 2.
    		
    		Applying Row's theorem, in \cite{10}, the authors define a
    		special type of graphs with two parallel paths, name zigzag
    		graphs, and show that a connected graph with $pt(G)=|G|-2$ is
    		a zigzag graph with special conditions.
    		
    		Of course the graphs are introduced in Theorem \ref{t2} and Corollary
    		\ref{pt}, are a special type of the zigzag graph and since every
    		graph with $Z_{c}(G)=2$ satisfies $Z(G)=2$, we can state a proof
    		based on these known results. But here we present an elementary
    		and straightforward proof for these theorems.
    	\end{remark}
%%%%%%%%%%%%%%%%%%%%%%%%%%%%%%%%%%%%%%%%%%%%
\bibliographystyle{amsplain}
%%%%%%%%%%%%%%%%%%%%%%%%%%%%%%%%%%%%%%%%%%%
%% Please cite your relevant papers but at most total 5 papers/books.
%- r_{0}((ab)^{\frac{1}{4}}-\sqrt{a})^{2}
%%%%%%%%%%%%%%%%%%%%%%%%%%%%%%%%%%%%%%%%%%%%

\end{document}